\documentclass[a4paper,11pt]{article}
\usepackage{amscd,amssymb,amsthm}
\usepackage{amssymb,palatino}
\usepackage{euscript}

\newtheorem{theorem}{Theorem}
\newtheorem{lemma}{Lemma}

\newtheorem{proposition}{Proposition}

\newtheorem{example}{Example}
\begin{document}

\title{On the Fiber Characters of $\mathbb F^*_{p^m}$ and \\
    related Polynomial Algebras         
 } 
\author{Michele Elia
\thanks{Politecnico di Torino
Corso Duca degli Abruzzi 24, I - 10129 Torino -- Italy; ~~ e-mail: elia@polito.it }}


\maketitle

\thispagestyle{empty}

\begin{abstract}
\noindent
Let $p$ be a prime, $m$ be a positive integer ( $m \geq 1$, and $m \geq 2$ if $p=2$), and $\chi_n$ be a multiplicative complex character
 on $\mathbb F^*_{p^m}$ with order $n| (p^m-1)$. We show that a partition
$\mathcal A_1 \cup \mathcal A_2 \cup \cdots \cup \mathcal A_n$
 of $\mathbb F^{*}_{p^m}$ is the partition by fibers of $\chi_n$ if and only if these fibers 
 satisfy certain additive properties. This is equivalent to showing that the set of multivariate characteristic polynomials of these
 fibers, completed with the constant polynomial $1$, is the basis of an $(n+1)$-dimensional commutative algebra with identity
  in the ring $\mathbb Q[x_1,\ldots,x_n]/\langle x_1^p-1, \ldots, x_n^p-1 \rangle$.
\end{abstract}

\vspace{1mm}

\vspace{1mm}

\section{Introduction}
In 1952, Perron gave some additive properties of the fibers of the
quadratic character on $\mathbb F_p$. Specifically in \cite{perron}, he showed that if
 $\mathfrak A, ~\mathfrak{B} \subset \mathbb F_p$ are the subsets of quadratic residues
  and non-residues, respectively, and letting 
 $ ~   d_p = \frac{p-1}{4}~~\mbox{if}~~  p =1\bmod 4, ~~ \mbox{and}~~ d_p = \frac{p+1}{4} ~~\mbox{if}~~ p=3 \bmod 4 ~$, then

\begin{enumerate}
   \item Every element of $\mathfrak A$ [respectively $\mathfrak B$] can be written as the sum of two elements of $\mathfrak A$
    [respectively $\mathfrak B$] in exactly $d_p - 1$ ways.
   \item  Every element of  $\mathfrak A$ [respectively $\mathfrak B$] can be written as the sum of
two elements of $\mathfrak B$ [respectively $\mathfrak A$] in exactly $d_p $ ways.
\end{enumerate}

\noindent
It was natural to inquire just how strong this result is, and to what extent it may hold for any
character $\chi_n$, other than $\chi_2$.
 In  \cite{monico}  it is shown that these additive properties uniquely characterize the even partition of $\mathbb F_p$
  into quadratic residues and non-residues. 
 In \cite{monico1}, the even restriction is removed, and the result is generalized to fibers of arbitrary
  multiplicative character $\chi_n$ on $\mathbb F_p$  ($n$ being a divisor of $(p-1)$), with suitable
 cyclotomic numbers in place of  the constants $d_p$ above. 
Lastly, in \cite{elia}, the generalization of the even partition (i.e. by the quadratic character $\chi_2$)
 to every finite field of odd characteristic, that is, the partition of $\mathbb F_{p^m}$ into squares
 and non-squares, is discussed and settled.  
Perron's view is attractive, but the formulation of the problem purely in terms of characteristic polynomials
 and their algebras permits a full description and proof of facts that occur in every finite field.
The purpose of this paper is to prove this definitive result.

\section{Preliminary results}
Let $\mathbb F_{p^m}$ be a finite field with  $p^m$ elements generated by a root $\gamma$ of a 
 primitive irreducible polynomial $p(x)=x^m+a_{m-1}x^{m-1}+\ldots+a_0$ over $\mathbb F_p$. 
Let $\mathcal B=\{1, \gamma, \ldots, \gamma^{m-1} \}$ be a basis of  $\mathbb F_{p^m}$,
any non-zero element $\beta \in \mathbb F_{p^m}$ is represented either as a power $\gamma^h$
 or in the basis $\mathcal B$ as $ \sum_{i=1}^m b_i \gamma^{i-1}$ with $\beta_i \in \mathbb F_p$. 
In the following, $\beta$ will be interchangeably indicated with the $m$-dimensional vector 
$\mathbf b= [ b_1,b_2, \ldots, b_m]  \in \mathbb F^m_p$, whenever necessary.  \\
A multiplicative complex character is
  an isomorphism $\chi: \mathbb F^*_{p^m} \rightarrow \mathcal C_{p^m-1}$ between the
 multiplicative cyclic group $\mathbb F^*_{p^m}$ and
  the complex multiplicative group $\mathcal C_{p^m-1}$ of the units of order $p^m-1$
 in the complex field $\mathbb C$. 
Let $n>1$ be a non-trivial positive divisor of $p^m-1$, that is 
$n \cdot s=p^m-1$ (if $p=2$ then $m$ must be greater than $1$), then the subset consisting 
of the powers of $\rho=\gamma^n$ is a cyclic subgroup of order $s$ of $\mathbb F^*_{p^m}$. \\
Let $\zeta_n$ be a primitive $n$th complex root of unity, i.e. $\zeta_n$ satisfies the
 $n$th cyclotomic polynomial. 
A character of order $n$ is explicitly defined as the mapping
 $\chi_n: \gamma  \rightarrow  e^{\frac{2\pi i}{n}} =\zeta_n$, that is
$$  \chi_n(\gamma^{n~t+h})   = e^{ \frac{2\pi i(n~t+h)}{n}}    = e^{\frac{2\pi i }{n}h}  = \zeta^h_n   ~~\forall t \in \mathbb Z
 ~~,~~   \mbox{and}~~  h \in \{0,1, \ldots, n-1 \}~~. $$

\noindent
For each integer $0 \leq k \leq (n-1)$ let $\mathcal A_{k}$ be the fiber  $\chi^{-1}( \zeta_n^{k-1})$, 
then the fiber $\mathcal A_1=\chi^{-1}(1)$ is the subgroup of $\mathbb F^*_{p^m}$
 consisting of the $n$-th powers of $\gamma$, and the fiber  $\mathcal A_k=\chi^{-1}( \zeta_n^{k-1})$,
 with $k>1$, is clearly the coset $\gamma^{k-1}\mathcal A_1$. 
 We have $|\mathcal A_k| = \frac{p^m-1}{n} = s$, and for each $1 \leq k \leq n$, the corresponding multivariate characteristic polynomial is
$$   q_k(\mathbf x)=q_k(x_1,\ldots,x_n) = \sum_{\beta \in \mathcal A_k}  \prod_{i=1}^m x^{b_i}_i    \in \mathbb Z[x_1,\ldots,x_n] ~~.    $$
The set of fibers $\mathcal A_1,\ldots ,\mathcal A_n$ form a partition of 
$\mathbb F^*_{p^m}=\{1, \gamma, \gamma^2,\ldots , \gamma^{p^m-2}\}$, thus, defining the polynomial 
$q_0(\mathbf x)=1$ which is the characteristic polynomial of the set $\{ 0 \}$,  we have
$$       \sum_{k=0}^n  q_k(\mathbf x) = \prod_{i=1}^{m} 	\frac{x_i^p-1}{x_i-1}  ~~.  $$ 
The following lemmas and theorem show that the set of these $n+1$ multivariate polynomials 
is the basis of an algebra of dimension $n+1$ in the polynomial ring 
$\mathfrak R_n[\mathbf x]=\mathbb Q[\mathbf x]/\langle x_1^p-1, \ldots, x_n^p-1 \rangle$,
where $\langle x_1^p-1, \ldots, x_n^p-1 \rangle$ denotes the ideal generated by the polynomials
 included in brackets.

\noindent
Since the fiber $\mathcal A_1$ is a sub-group of order $s=\frac{p^m-1}{n}$ of $\mathbb F^*_{p^m}$, 
and the remaining fibers are its cosets,
 which form a partition of $\mathbb F^*_{p^m}$, 
 the following proposition easily follows
\begin{proposition}
 \label{pro1}
 The set $\{q_0(\mathbf x), q_1(\mathbf x), \ldots, q_n(\mathbf x)\}$ of $n+1$ multivariate polynomials is a basis of a 
 $\mathbb Q$-subspace 
 $\mathbf V_{n+1}$ of dimension $n+1$ in the $p^m$-dimensional vector space 
 $\mathbb Q[\mathbf x]/\langle x_1^p-1, \ldots, x_n^p-1 \rangle$   of multivariate polynomials of degree
 at most $p - 1$ in each variable $x_i$. 
\end{proposition}

\noindent
The elements of $ \mathcal A_1$ have the following properties:

\begin{lemma}
   \label{lemma1}
Let $p$ be an odd prime, and assuming the above hypotheses, we have
\begin{enumerate}
\item If $s$ is even,  for any $\beta \in \mathcal A_1$ there exists a $\alpha \in \mathcal A_1$ 
  such that $\beta+\alpha=0$. 
\item If $s$ is odd, there exists a coset  $\mathcal E = \eta \mathcal A_1$ such that for any
    $\beta \in \mathcal A_1$ there is a $\alpha \in \mathcal E$ such that $\beta+\alpha=0$.
\end{enumerate}

\noindent
Let $p=2$, then
\begin{enumerate}
\item[3.] In $\mathbb F_{2^m}$, any element $\beta$ is the opposite of itself,
     i.e.  $\beta+\beta=0$.
\end{enumerate}
\end{lemma}

\begin{proof}
Consider the primitive element  $\gamma$ of $\mathbb F_{p^m}$, then
\begin{enumerate}
 \item If $s$ is even, the elements of $\mathcal A_1$ are all the roots of $X^s-1$,
  which splits as $(X^{s/2}-1)(X^{s/2}+1)$. Let $\eta=\gamma^n$ denote a root of $X^{s/2}+1$,
 and $\beta= \gamma^{2n t}$ be  any root of $X^{s/2}-1$. Since $\eta^{s/2}=-1$, we have 
$$  \beta \eta^{s/2} =-\beta= \gamma^{2n t} \gamma^{n s/2} = \gamma^{(2t+s/2)n} \in \mathcal A_1 ~~,$$
 therefore $ \beta+ \gamma^{(2t+s/2)n} =0$, i.e. $\alpha= \gamma^{(2t+s/2)n}$. 
 \item If $s$ is odd, no power of any element in $\mathcal A_1$ is equal to $-1$. However,
    let $\theta=\gamma^n$
  be a generator of the cyclic group $\mathcal A_1$, then an $\eta=\gamma^t \in \mathbb F_{p^m}$
   certainly exists such that $\theta+\eta=0$. Consider the coset $\eta \mathcal A_1$, therefore for any
 $\beta=\theta^u \in \mathcal A_1$, the element $\zeta=\eta\theta^{u-1}$ is such that $\beta+\zeta=0$
because we have
$$  \beta+\zeta=\theta^u+\eta\theta^{u-1} =\theta^{u-1} (\theta+\eta )= 0   ~, $$
 i.e.  $\alpha=\gamma^t \theta^{u-1} =\gamma^t \gamma^{n(u-1)}=\gamma^{t+n(u-1)}$.
  \item If $p=2$, then we trivially have $\beta+\beta=0$, thus in any fiber $\mathcal A_k$
       in $\mathbb F_{2^m}$,
     the sum of every element with itself is $0$, and the sum of two elements that are not
     in the same fiber is always different from zero. 
 \end{enumerate}
\end{proof}

\noindent
The immediate goal is to show that $\mathbf V_{n+1}$ is actually a $\mathbb Q$-sub-algebra 
of $\mathbb Q[\mathbf x]/ \langle x_1^{p} - 1, x_2^{p} - 1, \ldots , x_m^{p} - 1\rangle$. 
 
\begin{lemma}   
   \label{lemma2}
The following properties hold for the sums of elements of cosets in $\mathbb F^*_{p^m}$ with odd $p$:
\begin{enumerate}
\item If a fixed $u  \in \mathbb F_{p^m}$ can be expressed as the sum $\alpha_1+\alpha_2=u$,
 with $\alpha_1 \in \mathcal A_i$ and $\alpha_2 \in  \mathcal A_j$, then every element
 of the coset $\mathcal A_{k(u)}=u \mathcal A_1$ can be expressed as the sum of
 two elements, one from  $\mathcal A_i$, and one from  $\mathcal A_j$.
\item As a direct consequence of the previous point, the product $q_i(\mathbf x)q_j(\mathbf x)$
   is a linear combination of the basis elements of $\mathbf V_{n+1}$.
\end{enumerate}
\end{lemma}

\begin{proof}
The proof of claim 1 is immediate,  assuming $\alpha_1+\alpha_2=u$, we have
$$   \alpha (\alpha_1+\alpha_2) =\alpha \alpha_1+ \alpha \alpha_2 = \alpha u ~~\forall ~ \alpha \in 
       \mathcal A_1 ~~, $$
and the conclusion follows from the definition of the coset $\mathcal A_{k(u)}=u\mathcal A_1$, and group closure.

\vspace{2mm}

\noindent
The proof of claim 2 is a little more elaborate.  Due to the definition of the monomials
  $m(\mathbf x)$ that form part of the definition of the polynomials $q_k(\mathbf x)$, and the correspondence 
   $m(\mathbf x)\leftrightarrow \eta \in  \mathbb F_{p^m}$, the product of two monomials
   in the ring $\mathbb Q[\mathbf x]/ \langle x_1^{p} - 1, x_2^{p} - 1, \ldots , x_m^{p} - 1\rangle$ 
  corresponds to the sum of the corresponding elements in  $\mathbb F_{p^m}$. 
Now the product $q_i(\mathbf x)q_j(\mathbf x)$ consists of $s^2$ distinct monomials, which
 can be partitioned into groups of $s$ monomials, each group corresponding to
 some polynomial  $q_{k(i,j)}(\mathbf x)$ by the previous claim 1;
  the conclusion follows, by linearity.
\end{proof}

\begin{theorem}   
   \label{lemma3}
Let $2 \leq n|(p^m-1)$, $p$ prime, $m$ positive integer ($m\geq 2$ if $p=2$), and $s=\frac{p^m-1}{n}$.
The $\mathbb Q$-vector space $\mathbf V_{n+1}$ of Proposition \ref{pro1} is a
 $\mathbb  Q$-sub-algebra of the residue ring  
$\mathbb Q[\mathbf x]/ \langle x_1^{p} - 1, x_2^{p} - 1, \ldots , x_m^{p} - 1\rangle$.
 In particular, for every $1 \leq i, j \leq n$ there exist integers $c_{ijk}$ such that
\begin{equation}
   \label{eqsubalg}
   q_i(\mathbf x)q_j(\mathbf x) \bmod \langle x_1^{p} - 1, x_2^{p} - 1, \ldots , x_m^{p} - 1\rangle=c_{ij0}+
        \sum_{k=1}^n c_{ijk} q_k(\mathbf x) ~~.  
\end{equation}
The coefficients $c_{ij0}$ can be explicitly expressed considering $p$ odd and $p=2$ separately: 
\begin{itemize}
   \item[a)] $p$ odd
\begin{enumerate}
   \item  $c_{ii0} = s$ and $c_{ij0} = 0$ for every $j \neq i$ if $s$ is even;
   \item  $c_{ii0} = 0$ and $c_{ij0} = s$ for a suitable pair  $j \neq i$, if $s$ is odd. 
\end{enumerate}

\item[b)] $p=2$, in this case $s$ is always odd, and we have
\begin{enumerate}
   \item  $c_{ii0} = s$, and $c_{ij0} = 0$ for every $j \neq i$.
\end{enumerate}
\end{itemize}
\end{theorem}

\begin{proof}
The $\mathbb Q$-vector space $\mathbf V_{n+1}$ is a sub-algebra  
 of $\mathbb Q[\mathbf x]/ \langle x_1^{p} - 1, x_2^{p} - 1, \ldots , x_m^{p} - 1\rangle$  by Lemma \ref{lemma2}.

\vspace{3mm}
\noindent
In general it does not seem possible to obtain a closed form for all constants $c_{ijk}$
 holding for every $p$ and every $m$,  except for  the following exceptions. \\
  Let $\mathbf 1$ be the all-ones $m$-dimensional vector, then we have
$$   q_i(\mathbf 1)q_j(\mathbf 1)=c_{ij0}+\sum_{k=1}^n c_{ijk} q_k(\mathbf 1) ~~,  $$
since  $q_i(\mathbf 1)=s$, we have $s^2 = c_{ij0}+ s \sum_{k=1}^n c_{ijk}$; 
this equation implies that $s|c_{ij0}$;  since the integer $c_{ij0} \leq s$,
it follows that $c_{ij0}$ is either $0$ or $s$.
 If $p$ is odd, by Lemma \ref{lemma1} it follows that 
\begin{enumerate}
   \item  $c_{ii0} = s$ and $c_{ij0} = 0$ for every $j \neq i$ if $s$ is even;
   \item  $c_{ii0} = 0$ and $c_{ij0} = s$ for a suitable pair $j \neq i$, if $s$ is odd. 
\end{enumerate}

\noindent
If $p=2$, then $s$ is necessarily odd, however in $\mathbb F_{p^m}$ every element is
 the opposite of itself, then letting $m_\beta(\mathbf x)$ be the monomial associated to 
 $\beta$, it follows that 
$m_\beta(\mathbf x)^2 \bmod  \langle x_1^{2} - 1, x_2^{2} - 1, \ldots , x_m^{2} - 1\rangle $ is the monomial associated
 to $\beta+\beta=0$, that is the monomial $1$; it follows that 
\begin{enumerate}
   \item  $c_{ii0} = s$ and $c_{ij0} = 0$ for every $j \neq i$.
\end{enumerate}
\end{proof}

\noindent
Theorem \ref{lemma3} shows that the vector space $\mathbf V_{n+1}$ is a commutative sub-algebra 
 with identity of the ring of residue polynomials $\mathbb Q[\mathbf x]/ \langle x_1^{p} - 1, x_2^{p} - 1, \ldots , x_m^{p} - 1\rangle$. 
As observed in the proof of Theorem \ref{lemma3}, in general it seems that the structure constants 
cannot be given in closed form for every prime $p$, extension degree $m$, and power residue exponent $n$, 
thus the computational aspects for obtaining numerical values of every $c_{ijk}$ may be of interest.

\section{Computation of the structure constants}
The structure constants $c_{ijk}$ are easily found in closed form for $n=2$, $m=1$, and any odd $p$;
 however,  for every $m\geq 2$ and $n >2$, in general these constants must be numerically computed
 by means of convenient algorithms. We briefly, describe two different computational methods.
\subsection{Direct method}
For fixed $i,j$, equation (\ref{eqsubalg}) can be directly used to compute the structure constants.
A consistent linear system of $n$ equations in the $n$ unknowns $c_{ijk}$, $k=1, \ldots,n$, can be
 obtained by comparing the coefficients of equal multivariate monomials on the two sides 
 of  (\ref{eqsubalg}); actually, we would obtain a consistent linear system of $n^2$ linear equations in $n$ unknowns. 
 The search for the solution could present some difficulty because the product 
$$ q_1(x_1, \ldots, x_2)q_1(x_1, \ldots, x_2) \bmod \langle x_1^p-1, \ldots ,x_n^p-1 \rangle$$
consists of $n^2$ monomials in some order that, a priori, we do not know. They must all be computed,
 but only $n$ are used. 
When $n$ is small, as in the following examples, the method
 is very efficient, but when $n$ is large, $n^2$ multivariate monomials must be sorted according to 
 some ordering criterion:  this computational issue is left as an open problem.

\begin{example}
Let $p=3$ and $m=2$, thus $p^m-1=8$ and $n$ may be $2$ or $4$, which are the only proper divisors
  of $8$. 
Let $m(z)=z^2+z+2$ be a primitive quadratic polynomial over $\mathbb F_3$. Let $\alpha$ be a root
of $m(z)$, the $9$ elements of $\mathbb F_9$ are
$$   \begin{array}{|c|cr|} \hline
               0 &=  &  0   \\
               1 &=  &  1   \\
            \alpha & = & 0+ \alpha  \\
            \alpha^2 & = & 1+2\alpha  \\
            \alpha^3 & = & 2+2\alpha  \\
            \alpha^4 & = & 2  \\
          \alpha^5 & = & 0+2\alpha  \\
          \alpha^6 & = & 2+\alpha  \\
          \alpha^7 & = & 1+ \alpha  \\  \hline
      \end{array}
$$

\paragraph{Case 1: $ \mathbf n=2, ~s=4$}; we have two fibers (cosets) 
$$  \mathcal A_1=\{ 1,1+2\alpha, 2, 2+\alpha \}   ~~,~~
      \mathcal A_2=\{ \alpha, 2+2\alpha, 2\alpha,1+\alpha    \} ~,  $$ 
and the corresponding characteristic multivariate polynomials are
$$  q_1(x_1,x_2) = x_1+x_1x_2^2+x_1^2+x_1^2x_2 ~~~~,~~~~
      q_2(x_1,x_2) = x_2+x_1^2x_2^2+x_2^2+x_1x_2 ~~. $$
The structure constants of the polynomial algebra of $\mathbf V_3$, with basis
 $\{1,  q_1(x_1,x_2), q_2(x_1,x_2) \}$, are identified by the system
$$  \left\{  \begin{array}{lcl}
            q_1(x_1,x_2)q_1(x_1,x_2) \bmod \langle x_1^3-1,x_2^3-1 \rangle&=& c_{110}+c_{111}q_1(x_1,x_2)+c_{112}q_2(x_1,x_2) \\
            q_1(x_1,x_2)q_2(x_1,x_2) \bmod \langle x_1^3-1,x_2^3-1 \rangle&=& c_{120}+c_{121}q_1(x_1,x_2)+c_{122}q_2(x_1,x_2) \\
            q_2(x_1,x_2)q_2(x_1,x_2) \bmod \langle x_1^3-1,x_2^3-1 \rangle&=& c_{220}+c_{221}q_1(x_1,x_2)+c_{222}q_2(x_1,x_2) \\
    \end{array}   \right.
$$
where the constants with the third index equal to $0$ are known by Theorem \ref{lemma3}
  $c_{110}=4$, $c_{120}=0$, and $c_{220}=4$. \\
To find the remaining $6$ constants with
 the direct method we compute $q_i(x_1,x_2)q_j(x_1,x_2) \bmod \langle x_1^3-1,x_2^3-1 \rangle $
and subtract $(c_{ij1}q_1(x_1,x_2)+c_{ij2}q_2(x_1,x_2))+c_{ij0} $, 
obtaining three multivariate polynomials which must be identically zero
$$  \left\{  \begin{array}{lcl}
            4+(2-c_{112})x_1^2x_2^2+(1-c_{111})x_1^2x_2+(1-c_{111})x_1^2+(1-c_{111})x_1x_2^2+(2-c_{112})x_1x_2 & &  \\
~~~+(1-c_{111})x_1+(2-c_{112})x_2^2+(2-c_{112})x_2-4&=& 0 \\
            (2-c_{122})x_1^2x_2^2+(2-c_{121})x_1^2x_2+(2-c_{121})x_1^2+(2-c_{121})x_1x_2^2+(2-c_{122})x_1x_2 &  & \\
~~~~~+(2-c_{121})x_1+(2-c_{122})x_2^2+(2-c_{122})x_2&=& 0 \\
            4+(1-c_{222})x_1^2x_2^2+(2-c_{221})x_1^2x_2+(2-c_{221})x_1^2+(2-c_{221})x_1x_2^2+(1-c_{222})x_1x_2 & & \\
  ~~~~~   +(2-c_{221})x_1+(1-c_{222})x_2^2+(1-c_{222})x_2-4&=& 0 \\
    \end{array}   \right.
$$
From the first equation we obtain $c_{111}=1,~c_{112}=2$, from the second equation we obtain $c_{121}=2,~c_{122}=2$, and
  from the third equation $c_{221}=2,~c_{222}=1$, which allows us to write the multiplication table with the coefficients
  of the linear combinations (the trivial multiplications by $q_0(x_1,x_2)=1$ are not reported)
$$     \begin{array}{l|ccc}
                                                    & q_0(x_1,x_2)  & q_1(x_1,x_2)     &q_2(x_1,x_2)    \\ \hline
         q_1(x_1,x_2)q_1(x_1,x_2) &   4  &         1              &    2                  \\
         q_1(x_1,x_2)q_2(x_1,x_2) &  0  &         2               &   2                  \\
         q_2(x_1,x_2)q_2(x_1,x_2) &  4  &         2               &  1                  \\
       \end{array}
$$

\vspace{3mm}

\paragraph{Case 2: $\mathbf n=4, ~s=2$}; we have four cosets 
$$  \mathcal A_1=\{ 1, 2 \} ~~,~~\mathcal A_2=\{\alpha , 2\alpha \}   ~~,~~
      \mathcal A_3=\{2+ \alpha, 1+2\alpha\} ~~,~~\mathcal A_4=\{ 2+2\alpha,1+\alpha \} ~,  $$ 
and, correspondingly, the characteristic multivariate polynomials are
$$  q_1(x_1,x_2) = x_1+x_1^2 ~~,~~ q_2(x_1,x_2) = x_2+x_2^2 ~~,~~ q_3(x_1,x_2)=x_1x_2^2+x_1^2x_2
      ~~,~~ q_4(x_1,x_2) = x_1^2x_2^2+x_1x_2 ~~. $$
The multiplication table can be conveniently written as a $4\times 4$ table,
 where rows and columns are orderly indexed by 
  the polynomials $q_i(\mathbf x)$, and the entries are five-tuples of integers which are
  the five coefficients of the linear combinations
$$  \begin{array}{|c||c|c|c|c|} \hline
         &        q_1            & q_2             & q_3               &  q_4    \\   \hline \hline
  q_1 &     [2,1,0,0,0]     &  [0,0,0,1,1]  & [0,0,1,0,1]     &   [0,0,1,1,0]    \\   \hline
  q_2 &     [0,0,0,1,1]     &  [2,0,1,0,0]  & [0,0,1,0,1]     &   [0,0,1,1,0]    \\   \hline
  q_3 &     [0,0,1,0,1]     &  [0,0,1,1,0]  & [2,0,0,1,0]     &   [0,0,1,1,0]     \\   \hline
  q_4 &     [0,0,1,1,0]     &  [0,0,1,1,0]  & [0,0,1,0,1]     &   [2,0,0,0,1]     \\   \hline
       \end{array}
$$
For instance we have
$$  \left\{  \begin{array}{lcl}
 q_1(x_1,x_2)q_1(x_1,x_2) \bmod \langle x_1^3-1,x_2^3-1 \rangle&=&2 +q_1(x_1,x_2) \Rightarrow  [2,1,0,0,0] \\
q_1(x_1,x_2)q_2(x_1,x_2) \bmod \langle x_1^3-1,x_2^3-1 \rangle&=&q_3(x_1,x_2)+q_4(x_1,x_2) \Rightarrow  [0,0,0,1,1]
    \end{array}   \right.
$$
\end{example}

\vspace{3mm}

\begin{example}
Let $m(z)=z^4+z+1$ be a $4$-degree primitive polynomial over $\mathbb F_2$. Let $\alpha$ be a root
of $m(z)$, the $16$ elements of $\mathbb F_{16}$ are
$$   \begin{array}{|c|cr|} \hline
               0 &=  &  0   \\
               1 &=  &  1   \\
            \alpha & = & \alpha  \\
            \alpha^2 & = &  \alpha^2  \\
            \alpha^3 & = &  \alpha^3  \\
            \alpha^4 & = & 1+\alpha  \\
          \alpha^5 & = & \alpha+\alpha^2  \\
          \alpha^6 & = & \alpha^2+\alpha^3  \\
          \alpha^7 & = & 1+ \alpha+\alpha^3  \\ 
            \alpha^8 & = & 1+\alpha^2  \\
            \alpha^9 & = &\alpha+  \alpha^3  \\
            \alpha^{10} & = & 1+\alpha+ \alpha^2  \\
            \alpha^{11} & = & \alpha+\alpha^2+\alpha^3  \\
          \alpha^{12} & = & 1+\alpha+\alpha^2+\alpha^3  \\
          \alpha^{13} & = & 1+ \alpha^2+\alpha^3  \\
          \alpha^{14} & = & 1+\alpha^3  \\  \hline
      \end{array}
$$

\noindent
In this case $n$ may be $3$ or $5$; only $n=3$ is considered, being fully illustrative.

\paragraph{Case: $n=3$, $s=5$}; we have three cosets
$$  \left\{  \begin{array}{lcl}
  \mathcal A_1 &= &\{ 1,\alpha^3, \alpha^2+\alpha^3, \alpha+  \alpha^3, 1+\alpha+\alpha^2+\alpha^3 \} \\ 
   \mathcal A_2 &=& \{ \alpha,1+\alpha , 1+ \alpha+\alpha^3, 1+\alpha+\alpha^2,1+ \alpha^2+\alpha^3\} \\
   \mathcal A_3 &=& \{ \alpha^2, \alpha+\alpha^2, 1+\alpha^2,\alpha+\alpha^2+\alpha^3, 1+\alpha^3 \}
    \end{array}  \right.
  $$ 
and correspondingly three characteristic multivariate polynomials
$$   \left\{  \begin{array}{lcl}
    q_1(x_1,x_2,x_3,x_4) &= & x_1+ x_4 + x_3 x_4+x_2x_4+x_1x_2x_3 x_4   \\ 
     q_2(x_1,x_2,x_3,x_4) &=&  x_2+x_1x_2+x_1x_2x_4+x_1x_2x_3+x_1x_3x_4  \\
     q_3(x_1,x_2,x_3,x_4) &=&  x_3+ x_2x_3+ x_1x_3+ x_2x_3x_4+ x_1x_4  
    \end{array}  \right.~~. $$
Let $\mathbf x=[x_1,x_2,x_3,x_4]$, a basis of $\mathbf V_4$ is $\{1,  q_1(\mathbf x), q_2(\mathbf x),q_3(\mathbf x) \}$,
and the structure constants of the polynomial algebra can be computed from the
 following system of six equations
$$  \left\{  \begin{array}{lcl}
            q_1(\mathbf x)q_1(\mathbf x)&=& c_{110}+c_{111}q_1(\mathbf x)+c_{112}q_2(\mathbf x) +c_{113}q_3(\mathbf x) \\
            q_1(\mathbf x)q_2(\mathbf x)&=& c_{120}+c_{121}q_1(\mathbf x)+c_{122}q_2(\mathbf x) +c_{123}q_3(\mathbf x) \\
            q_1(\mathbf x)q_3(\mathbf x)&=& c_{130}+c_{131}q_1(\mathbf x)+c_{132}q_2(\mathbf x) +c_{133}q_3(\mathbf x) \\
          q_2(\mathbf x)q_2(\mathbf x)&=& c_{220}+c_{221}q_1(\mathbf x)+c_{222}q_2(\mathbf x) +c_{223}q_3(\mathbf x) \\
          q_2(\mathbf x)q_3(\mathbf x)&=& c_{230}+c_{231}q_1(\mathbf x)+c_{232}q_2(\mathbf x) +c_{233}q_3(\mathbf x) \\
          q_3(\mathbf x)q_3(\mathbf x)&=& c_{330}+c_{331}q_1(\mathbf x)+c_{332}q_2(\mathbf x) +c_{333}q_3(\mathbf x) \\
    \end{array}   \right.
$$
Now $c_{110}=c_{220}=c_{330}=5$, and $c_{120}=c_{130}=c_{230}=0$,  then we have to compute only
 $18$ constants instead of $24$.
Proceeding as in the previous example  we obtain all structure constants $c_{ijk}$ and write 
the multiplication table where the coefficients
  of the linear combinations for $ q_i(\mathbf x)q_j(\mathbf x)$ are reported
 in the corresponding row (the trivial multiplications by $q_0(x_1,x_2)=0$ are not reported)
$$     \begin{array}{l|cccc}
                                                    & q_0(\mathbf x)  & q_1(\mathbf x)     &q_2(\mathbf x)  &q_3(\mathbf x)  \\ \hline
         q_1(\mathbf x)q_1(\mathbf x) &  5  &         0               &   2        &     2     \\
         q_1(\mathbf x)q_2(\mathbf x) &  0  &         2               &   2        &      1    \\
         q_1(\mathbf x)q_3(\mathbf x) &  0  &         2               &   1        &      2    \\
         q_2(\mathbf x)q_2(\mathbf x) &  5  &         2               &   0        &      2     \\
         q_2(\mathbf x)q_3(\mathbf x) &  0  &         1               &   2        &      2      \\
         q_3(\mathbf x)q_3(\mathbf x) &  5  &         2               &   2        &      0     \\
      \end{array}
$$

\end{example}

\subsection{A numerical method based on cyclotomic fields}
Let $\mathbb Q(\zeta_{p})$ be the cyclotomic field of $p$-th roots of unity, with $\zeta_{p}$ 
 denoting a primitive root of unity, that is a root of the cyclotomic polynomial  of
 degree $p-1$. Thus $\mathbb Q(\zeta_{p})$ is an extension of degree $p-1$ of 
 $\mathbb Q$.
Let $\mathfrak G_p$ denote the multiplicative cyclic group generated by $\zeta_{p}$. 
Let $\mathbf u =(\zeta^{i_1}_p,\zeta^{i_2}_p, \ldots ,\zeta^{i_m}_p)$ denote an $n$-tuple
  of elements of $\mathfrak G_p$,  thus from the evaluation of equation (\ref{eqsubalg}) for
$\mathbf x =\mathbf u $, we get a polynomial in $\zeta_p$ that is equal to $0$ 
\begin{equation} 
   \label{nummeq}
   c_{ij0}+
        \sum_{k=1}^n c_{ijk} q_k(\mathbf u) - q_i(\mathbf u)q_j(\mathbf u) =0~~.    
\end{equation}
We thus obtain a system of $p$ linear equations with integral coefficients in 
 $n$ unknowns.
If $n\leq p$ a solution is easily obtained, since it certainly exists by Theorem \ref{lemma3}.
If $n>p$ we need more linear equations, then we consider the equations obtained
 using $\ell$ different vectors $\mathbf u$, with the
aim of getting $n$ linearly independent equations.

\begin{example}
Reconsider the problem of example 1. Its solutions by this second method are
  obtained working in  $\mathbb Q(\zeta_{3})$ with $\zeta_3$ a
 primitive complex cubic root of unity. \\
 Take $\mathbf u_0=(1,1)$, and $\mathbf u_1=(1,\zeta_3)$; in this case
 we obtain two equations using (\ref{nummeq}), considering that 
 $c_{110}=4$, $c_{120}=0$, and $c_{220}=4$,  
 $ q_1(1,1)=q_2(1,1)=4$,  $ q_1(1,\zeta_3)=1$, and $q_2(1,\zeta_3)=-2$.
Thus we can write the system
$$    \left\{  \begin{array}{lcl}
            4+4c_{111}+4c_{112} &=& 16 \\
             4+c_{111}-2c_{112}  &=& 1
         \end{array}  \right.
$$
Solving for $ c_{111}, c_{112}$ we obtain  $ c_{111}=1, c_{112}=2$.   \\
Similarly, we obtain all structure constants summarized in the following table 
$$     \begin{array}{l|ccc}
                                                    & 1  & q_1(x_1,x_2)     &q_2(x_1,x_2)    \\ \hline
         q_1(x_1,x_2)q_1(x_1,x_2) &   4  &         1              &    2                  \\
         q_1(x_1,x_2)q_2(x_1,x_2) &  0  &         2               &   2                  \\
         q_2(x_1,x_2)q_2(x_1,x_2) &  4  &         2               &  1                  \\
       \end{array}
$$
\end{example}

\vspace{5mm}

\subsection{A new proof of Perron's original observations}  
The history of the $\mathbb F^*_{p^m}$ partition by the fibers of a given character 
 began with Perron's characterization of the sets or quadratic residues and non-residues 
in prime fields, and several independent proofs have since been given.  A "new" proof is
obtained by specializing the general results given above, and holds for every finite field of odd characteristic . \\
Consider the prime field $\mathbb F_{p^m}$, $p$ odd, and the character $\chi_2$ of order $2$
defined over $\mathbb F^*_{p^m}$.
  Let $\mathcal R$ and $\mathcal N$ be the subsets of
$\mathbb F^*_{p^m}$ of squares and non-squares, respectively, that is
 $\mathcal R=\chi^{-1}(1)$ and $\mathcal N=\chi^{-1}(-1)$.
The corresponding characteristic polynomials are
$$
  q_{\mathcal R}(\mathbf x) = \sum_{\beta\in \mathbb F^*_{p^m}} \frac{1+\chi_2(\beta)}{2} \prod_{i=1}^m x_i^{b_i}  
   ~~,~~ q_{\mathcal N}(x) = \sum_{\beta\in \mathbb F^*_{p^m}} \frac{1-\chi_2(\beta)}{2} \prod_{i=1}^m x_i^{b_i}  ~~,  $$
 depending on whether $\frac{p^m-1}{2}$  is odd or even, we have
\begin{equation}
   \label{eqmain}
\begin{array}{l}
\frac{p^m-1}{2} ~~ \mbox{odd}~~~~
  \left\{  \begin{array}{lcl}
            q_{\mathcal R}(\mathbf x)q_{\mathcal R}(\mathbf x) \bmod \langle x_1^p-1, \ldots,x_m^p-1\rangle&=&~~~~~0+a_{11} q_{\mathcal R}(\mathbf x)+b_{11} q_{\mathcal N}(\mathbf x) \\
q_{\mathcal R}(\mathbf x)q_{\mathcal N}(\mathbf x) \bmod \langle x_1^p-1, \ldots,x_m^p-1\rangle&=&\frac{p-1}{2}+a_{12} q_{\mathcal R}(\mathbf x)+b_{12} q_{\mathcal N}(\mathbf x) \\
 q_{\mathcal N}(\mathbf x)q_{\mathcal N}(\mathbf x) \bmod \langle x_1^p-1, \ldots,x_m^p-1\rangle&=&~~~~~0+a_{22} q_{\mathcal R}(\mathbf x)+b_{22} q_{\mathcal N}(\mathbf x) \\
    \end{array}   \right. \\
\\
 \frac{p^m-1}{2} ~~ \mbox{even}~~~~  \left\{  \begin{array}{lcl}
            q_{\mathcal R}(\mathbf x)q_{\mathcal R}(\mathbf x) \bmod \langle x_1^p-1, \ldots,x_m^p-1\rangle&=&\frac{p-1}{2}+a_{11} q_{\mathcal R}(\mathbf x)+b_{11} q_{\mathcal N}(\mathbf x) \\
q_{\mathcal R}(\mathbf x)q_{\mathcal N}(\mathbf x) \bmod \langle x_1^p-1, \ldots,x_m^p-1\rangle&=&~~~~~0+a_{12} q_{\mathcal R}(\mathbf x)+b_{12} q_{\mathcal N}(\mathbf x) \\
 q_{\mathcal N}(\mathbf x)q_{\mathcal N}(\mathbf x) \bmod \langle x_1^p-1, \ldots,x_m^p-1\rangle&=&\frac{p-1}{2}+a_{22} q_{\mathcal R}(\mathbf x)+b_{22} q_{\mathcal N}(\mathbf x) \\
    \end{array}   \right.
\end{array}
\end{equation}

\noindent
Let $\mathbf u_o$ be the vector of all ones, then we have
$$   q_{\mathcal R}(\mathbf u_o) = \sum_{\beta\in \mathbb F^*_{p^m}} \frac{1+\chi_2(\beta)}{2}  =
 \frac{p^m-1}{2}  ~~,~~q_{\mathcal N}(\mathbf u_o) = \sum_{\beta\in \mathbb F^*_{p^m}} \frac{1-\chi_2(\beta)}{2}  =   \frac{p^m-1}{2}    $$
$$   q_{\mathcal R}(-\mathbf u_o) = \sum_{\beta\in \mathbb F^*_{p^m}} \frac{1+\chi_2(\beta)}{2} (-1)^{\sum_{i=1}^m b_i} =
t ~,~q_{\mathcal N}(-\mathbf u_o) = \sum_{\beta\in \mathbb F^*_{p^m}} \frac{1-\chi_2(\beta)}{2} 
(-1)^{\sum_{i=1}^m b_i} =   -t  $$
If $q_{\mathcal R}(-\mathbf u_o)=0$, it is necessary to use a vector $\mathbf u$ different from $-\mathbf u_o$: there are $2^m-2$ possible choices for $\mathbf u \neq -\mathbf u_o$, and one of them certainly works because of Theorem \ref{lemma3}. 

\end{document}